\documentclass[a4paper, 12pt]{amsart}

\usepackage{tikz}
\usetikzlibrary{matrix}
\usepackage{tikz,tikz-cd}
\usepackage{amsrefs}
\usepackage{amsmath}
\usepackage{hyperref}
\usepackage{amssymb}
\usepackage{latexsym}
\usepackage{amscd}
\usepackage{graphicx} 
\usepackage{amsthm}
\usepackage{mathrsfs}
\usepackage{xypic}
\usepackage{bm}
\usepackage{color}
\usepackage{ulem}

\newdimen\AAdi%
\newbox\AAbo%
\def\AAk#1#2{\s_etbox\AAbo=\hbox{#2}\AAdi=\wd\AAbo\kern#1\AAdi{}}%
\def\AAr#1#2#3{\s_etbox\AAbo=\hbox{#2}\AAdi=\ht\AAbo\raise#1\AAdi\hbox{#3}}%
\font\tenmsb=msbm10 at 12pt \font\sevenmsb=msbm7 at 8pt
\font\fivemsb=msbm5 at 6pt
\newfam\msbfam
\textfont\msbfam=\tenmsb \scriptfont\msbfam=\sevenmsb
\scriptscriptfont\msbfam=\fivemsb

\newtheorem{thm}{Theorem}[section]

\newtheorem{rem}[thm]{Remark}
\newtheorem{pro}[thm]{Proposition}
\newtheorem{defi}[thm]{Definition}

\newcommand{\Section}[2]{\setcounter{equation}{0}
\allowdisplaybreaks
\section[#1]{#2}}

\def\f#1#2{\frac{#1}{#2}}

\def\p#1{\partial #1}

\def\la{\lambda}

\def\Om{\Omega}

\def\Re{\text{Re }}
\def\Im{\text{Im }}

\subjclass[2010]{58E20,~53A10,~53C42.}
\begin{document}
\pagenumbering{Roman}\setcounter{page}{1}

\pagenumbering{arabic} \setcounter{page}{1}
\title{On instability of Type (II) Lawson-Osserman Cones}
\author{Zhaohu\ Nie}
\address{
            Department of Mathematics and Statistics,
            Utah State University,
            Logan, UT 84322-3900,
            USA}
            \email
 {
zhaohu.nie@usu.edu
}

\author{Yongsheng\ Zhang}
\address{
            School of Mathematical Sciences and Institute of Advanced Study, Tongji University, Shanghai, 200092, China}\email
 {
yongsheng.chang@gmail.com
}

\date{}


\begin{abstract}

We obtain the instability of Type (II) Lawson-Osserman cones in Euclidean spaces, and thus
provide a family of (uncountably many) 
unstable solutions 
with singularity to the Dirichlet problem for minimal graphs of  high codimension
versus smooth unstable ones by Lawson-Osserman 
using a min-max technique.
To our knowledge,
these are the first examples of non-smooth unstable minimal graphs
and unlikely detectible by the mean curvature flow or min-max theory.

\end{abstract}
\maketitle
\renewcommand{\proofname}{\it Proof.}
\Section{Introduction}{Introduction}
Given an open, bounded, and strictly convex $\Omega\subset\mathbb R^{n+1}$ $(n\geq 1)$ and a continuous map $\phi:\p \Omega\rightarrow {\mathbb R}^{m+1}$ (called the boundary data),
           the \textbf{Dirichlet problem} (cf. \cite{j-s,b-d-m, de, m1,l-o}) 
           searches for $\Phi\in C^0(\bar \Omega)\cap \mbox{Lip}(\Omega)$ taking values in ${\mathbb R}^{m+1}$ such that
 %
the graph of $\Phi$ is a minimal submanifold in $\mathbb R^{m+n+2}$
and
         $
          \Phi|_{\partial \Omega}=\phi.
          $

When $m=0$, 
     for any continuous boundary data 
                    there exists a unique Lipschitz solution
       according to J. Douglas \cite{d},
                    T. Rad\'o \cite{r,r2},
                    Jenkins-Serrin \cite{j-s}
                    and Bombieri-de Giorgi-Maranda \cite{b-d-m}.                     
                    Furthermore,  
                    the solution is in fact analytic
                    by E. de Giorgi \cite{de} and J. Moser \cite{m1}, 
                    and its graph turns out to be area-minimizing, e.g. see \cite{fe}.

When $m\geq 1$, situations are completely different.
          With $\Om=\mathbb B^{n+1}$ (the unit ball, the case that we shall consider in this paper),
          Lawson-Osserman \cite{l-o}
          constructed
          real analytic boundary data 
          for $n,m\geq 1$ for which there exist at least three analytic solutions;
          boundary data for which the problem is not solvable
          for $n\geq m+1\geq 3$;
          and boundary data that support
          Lipschitz but non-$C^1$ solutions.
          
          Recently 
systematic developments on Lawson-Osserman constructions in \cite{l-o} 
have been made in \cite{x-y-z0, z}. 
In particular, uncountably many boundary data are discovered in \cite{x-y-z0}, 
         each of which supports
          infinitely many analytic solutions and at least one singular solution. 
          The graph of such a singular solution is just the cone 
          over the graph of the boundary data, and it is
          called a Type {\bf (II)} Lawson-Osserman cone.
          In more details, 
          the boundary data are suitable multiples of 
          Type {\bf (II)} LOMSEs (see Definition \ref{LOMSE} below) between unit spheres
          and
          the corresponding countably many analytic solutions have graphs with increasing volumes.
          The limit is the volume of the truncated 
          Lawson-Osserman cone.
          Therefore, such cone is not area-minimizing.
         Since these analytic solutions do not form a continuous family  
         around the Lawson-Osserman cone yet,
         it remains open whether the cone is stable or not. 
         In this paper, we settle this question.
         
          \begin{thm}\label{t1}
          Type {\bf (II)} Lawson-Osserman cones are all unstable.
          \end{thm}
          
          The idea is to study a corresponding quotient space endowed with a canonical metric.
          Inspired by \cite{b, h-l, l} for orbit spaces, 
          we focus on a preferred subspace $W$ (by \ref{W}) associated to the given LOMSE
          and the quotient space of the subspace.
          With canonical metric, the length of any curve in the quotient space 
          equals the volume of the corresponding submanifold in the Euclidean space.
          Hence, the above solutions induce geodesics in the quotient space
          connecting a fixed point (associated to the LOMSE) and the origin (a boundary point of the quotient space).
          Consequently, we have a LOC curve standing for the entire Lawson-Osserman cone and
          a LOC segment 
          for the truncated part respectively.

                      Note that the instability of LOC segments for Type {\bf (II)} naturally implies
          the instability of the truncated Lawson-Osserman cones.
          However, it is not simple to gain stability for Type {\bf (I)} in Euclidean spaces 
          (besides those area-minimizing LOCs of $(n,p,2)$-type proven in \cite{x-y-z}).
          In this paper, we show that 
          an LOC segment for  
          %
          Type {\bf (I)} 
          turns out to be locally length-mimimizing 
          (namely length-minimizing in certain angular sector that contains the LOC curve)
          as announced in \cite{zh}.
          As a result, we have
          
          \begin{thm}\label{t2}
          LOC segments for Type {\bf (I)} Lawson-Osserman cones are all stable.
          \end{thm}
          
          The paper is organized as follows.
         We shall briefly review
         Lawson-Osserman cones in \S\ref{S2}.
         The construction of canonical metric on quotient space mentioned above will be provided in \S\ref{S3}.
          It will be verified explicitly in \S\ref{S4} that the geodesic equation is equivalent to the minimality requirement\eqref{ODE1} of corresponding submanifold represented.
         Section \S\ref{S5} will be devoted to computations of instability of Type {\bf (II)} Lawson-Osserman curves on quotient spaces and stability of Type {\bf (I)}.
         We further explain in \S\ref{S6} 
         the interesting translation from behavior around the spiral stable fixed point of \eqref{ODE2} for Type {\bf (II)} 
         to Jacobi fields along the LOC curve.
         In \S\ref{S7} the construction of certain calibration forms can be achieved for the local length-minimality on the quotient space for Type {\bf (II)}.
          
   \noindent{\bf Acknowledgment.}       The authors would like to thank MPIM for warm hospitality and financial supports,  	
   where they conducted the research during their visits in Spring 2019. 
               The research of Z.N. is partially supported by the Simons Foundation through Grant \#430297. 
               The research of Y.Z. is sponsored in part by NSFC (Grant Nos. 11971352,\ 11601071), 
the S.S. Chern Foundation (through IHES),
and a Start-up Research Fund from Tongji University.
{\ }
\Section{Background on LOMSE}{Background on LOMSE}\label{S2}
Let us recall some materials from \cite{x-y-z0}.

\begin{defi}\label{loc}
For a smooth map $f: S^{n}\rightarrow S^{m}$,
            if there exists an acute angle $\theta$, such that 
        \begin{equation}\label{lo}
        M_{f,\theta} :=\left\{(\cos\theta\cdot x,\sin\theta\cdot f(x)):x\in S^{n}\right\}
        \end{equation}
is a minimal submanifold of $S^{m+n+1}$, then
$f$ is called a {\bf Lawson-Osserman map} (\text{LOM}),
$M_{f,\theta}$ the associated {\bf Lawson-Osserman sphere} (\text{LOS}),
and the cone $C_{f,\theta}$ over $M_{f,\theta}$ 
the associated {\bf Lawson-Osserman cone} (\text{LOC}).
\end{defi}
For $\phi=\tan\theta\cdot f$,
the cone over the graph of $\phi$ is minimal and determines
a Lipschitz but non-$C^1$ solution to the Dirichlet problem.
Therefore, it is important to establish a characterization of LOM.
Let $g$ be the induced metric on $S^n$ via \eqref{lo} from the standard metric $g_{m+n+1}$
of  $S^{m+n+1}$.

\begin{thm}[Characterization of LOM \cite{x-y-z0}]
For smooth $f:S^n\rightarrow S^m$ and $\theta\in (0,\pi/2)$,
$M_{f,\theta}$ is an LOS in $S^{n+m+1}$
if and only if the following conditions hold:
              \begin{itemize}
\item  $f:(S^n,g)\rightarrow (S^m,g_m)$ is harmonic.
\item For each $x\in S^n$ and the singular values $\lambda_1,\cdots,\lambda_n$  of
$(f_*)_x: (T_x S^n,g_n)\rightarrow (T_{f(x)} S^m,g_m)$, 
               namely $\{\lambda_i\}$ are the diagonals of $(f_*)^*g_m$ with respect to $g_n$,
we have
             \begin{equation}\label{sv}
             \sum\limits_{j=1}^n \frac{1}{\cos^2\theta+\sin^2\theta \lambda_j^2}=n.
             \end{equation}
              \end{itemize}
\end{thm}

        Since the second condition is in general hard to interpret,
        \cite{x-y-z0} restricts discussions to a special case.
   
   \begin{defi}   \label{LOMSE}
        If $f$ is an LOM and in addition, for each $x\in S^n$, all nonzero singular values of $(f_*)_x$ are equal, i.e.,
$$\{\lambda_1,\cdots,\lambda_n\}=\{0,\lambda\},$$
then it is called an {\bf LOMSE}.
     \end{defi}

Note that the distribution of singular values (counting multiplicities) has to be the same pointwise.
          Moreover, LOMSEs have a very pleasant structure decomposition.
          \begin{thm}[Structure of LOMSE \cite{x-y-z0}]\label{Str}
 Smooth $f$ between unit spheres is an LOMSE 
          with $\{\lambda_1,\cdots,\lambda_n\}=\{0,\lambda\}$ 
          if and only if
           $f=i\circ\pi$
           where $\pi$ is a Hopf fibration from $S^n$ to a (complex, quaternionic, or octonionic) projective space $(\mathbb P^p,h)$ of real dimension $p$ 
           and $i$ is a minimal isometric immersion: $(\mathbb P^p,\lambda^2h) 
           {\looparrowright} (S^m,g_m)$.
\end{thm}
\begin{rem}\label{evenk}
As a map to Euclidean space, an LOMSE has $(m+1)$-components
and all of them are spherical harmonic polynomials of even degree $k$.
We call such LOMSE and associated LOC of $\mathbf{(n,p,k)}${\bf -type}.
Furthermore, \cite[eq. (2.26)]{x-y-z0} gives
\begin{equation}\label{la}
\lambda=\sqrt{\frac{k(k+n-1)}{p}}.
\end{equation}
\end{rem}

          Different solutions to the Dirichlet problem are obtained by considering 
\begin{equation}\label{Mrf}         
          M=M_{\rho,f}:=
               \left\{(rx,\rho(r)f(x)):x\in S^n, r\in 
               (0,\infty)
               \right\}
               \subset \mathbb R^{m+n+2}
\end{equation}
          for analytic solutions. 
          A characterization of $M$ being minimal is the following. 
                    \begin{thm}[Evolution Equation \cite{x-y-z0}]
           For LOMSE $f$, $M$ is a minimal graph
if and only if  
                    \begin{equation}\label{ODE1}
\frac{\rho_{rr}}{1+\rho_r^2}+\frac{(n-p)\rho_r}{r}+\frac{p(\frac{\rho_r}{r}-\frac{\la^2\rho}{r^2})}{1+\frac{\la^2\rho^2}{r^2}}=0.
\end{equation}
           \end{thm}
           
           There are two types
           \begin{itemize}
           \item [\bf {\bf (I)}]
     $(n,p,k)=(3,2,2), (5,4,2), (5,4,4)$ or $n\geq 7$, and
          \item [\bf (II)] 
          $(n,p)=(3,2)$, $k\geq 4$ or $(n,p)=(5,4)$, $k\geq 6$;
          \end{itemize}
          for which solutions to \eqref{ODE1}
           emitting from the origin behave differently.


                              $$\begin{minipage}[c]{0.55\textwidth}
                              \includegraphics[scale=0.42]{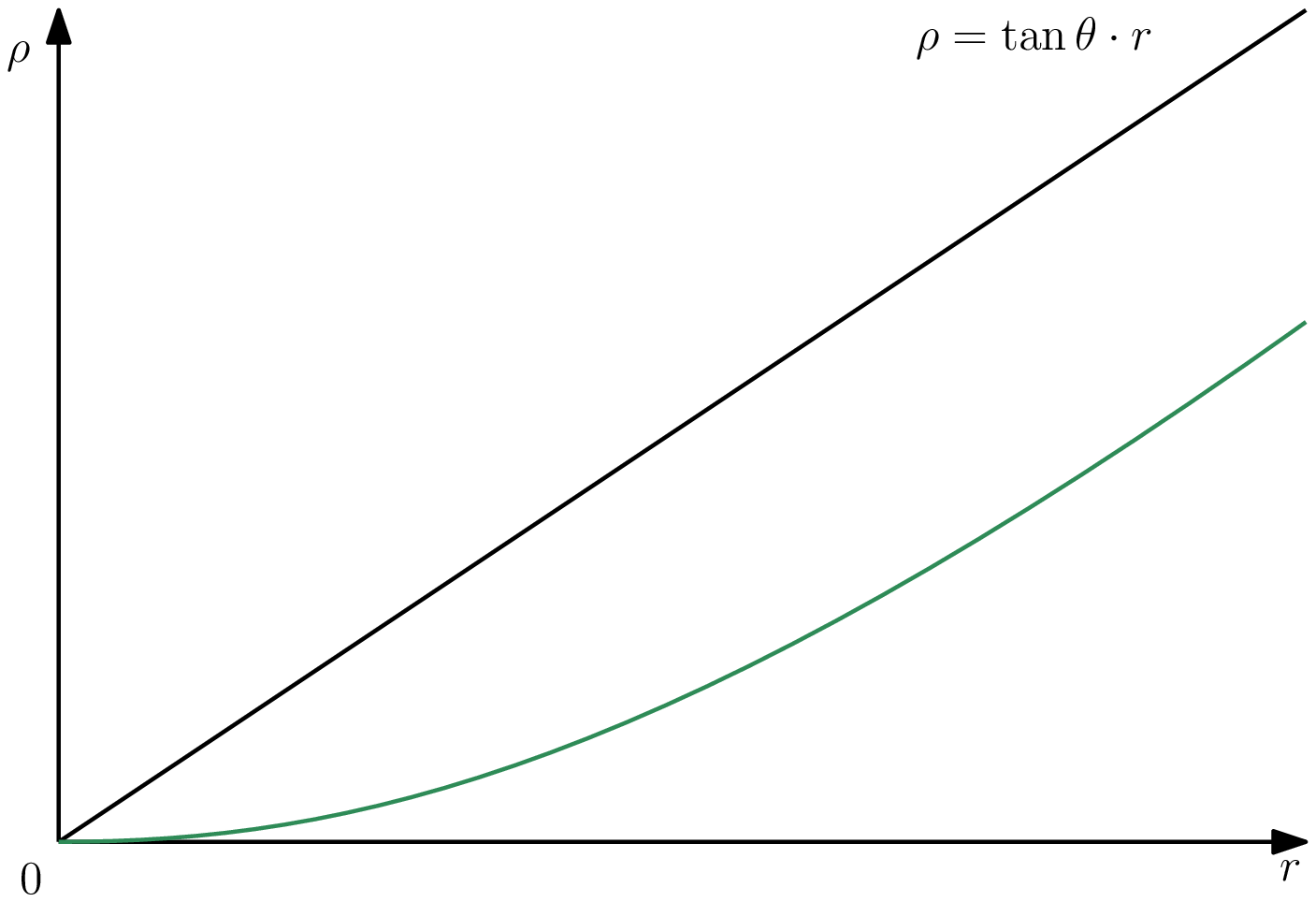}
                              \end{minipage}%
                          \begin{minipage}[c]{0.51\textwidth}
                           \includegraphics[scale=0.448]{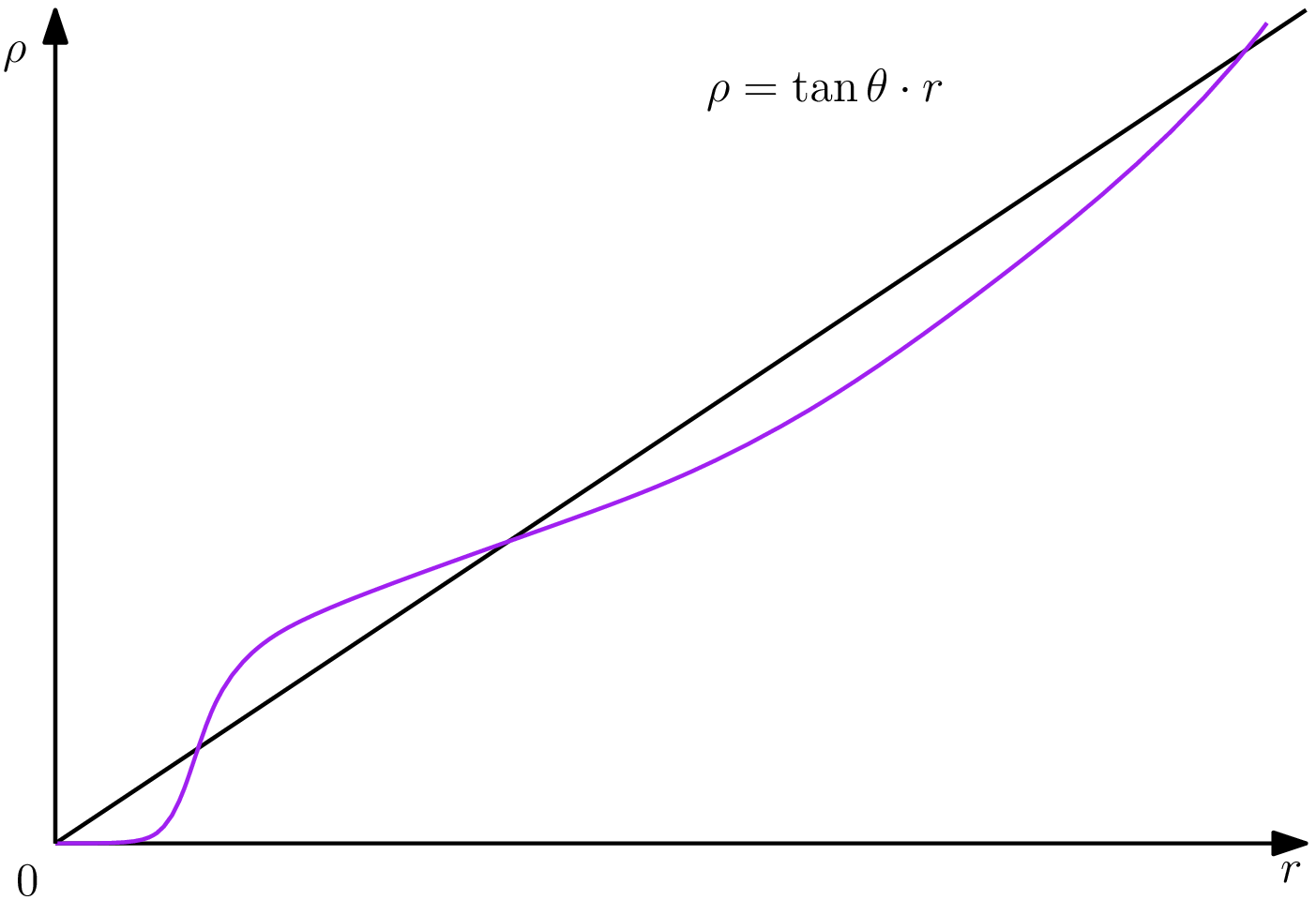}
                           \end{minipage}$$


        Here  the LOC
curve $\{\rho=\tan\theta\cdot r\}$ stands for the LOC.
          For Type {\bf (II)} LOMSE $f$,
          each intersection point of the oscillating curve with the LOC curve
          leads to a minimal graph over the unit ball $\mathbb B^{n+1}\subset \mathbb R^{n+1}$
          by rescaling the oscillating curve in different scales to connect the origin and $(1,\tan\theta)$ as follows
          \begin{figure}[h]
          \includegraphics[scale=0.48]{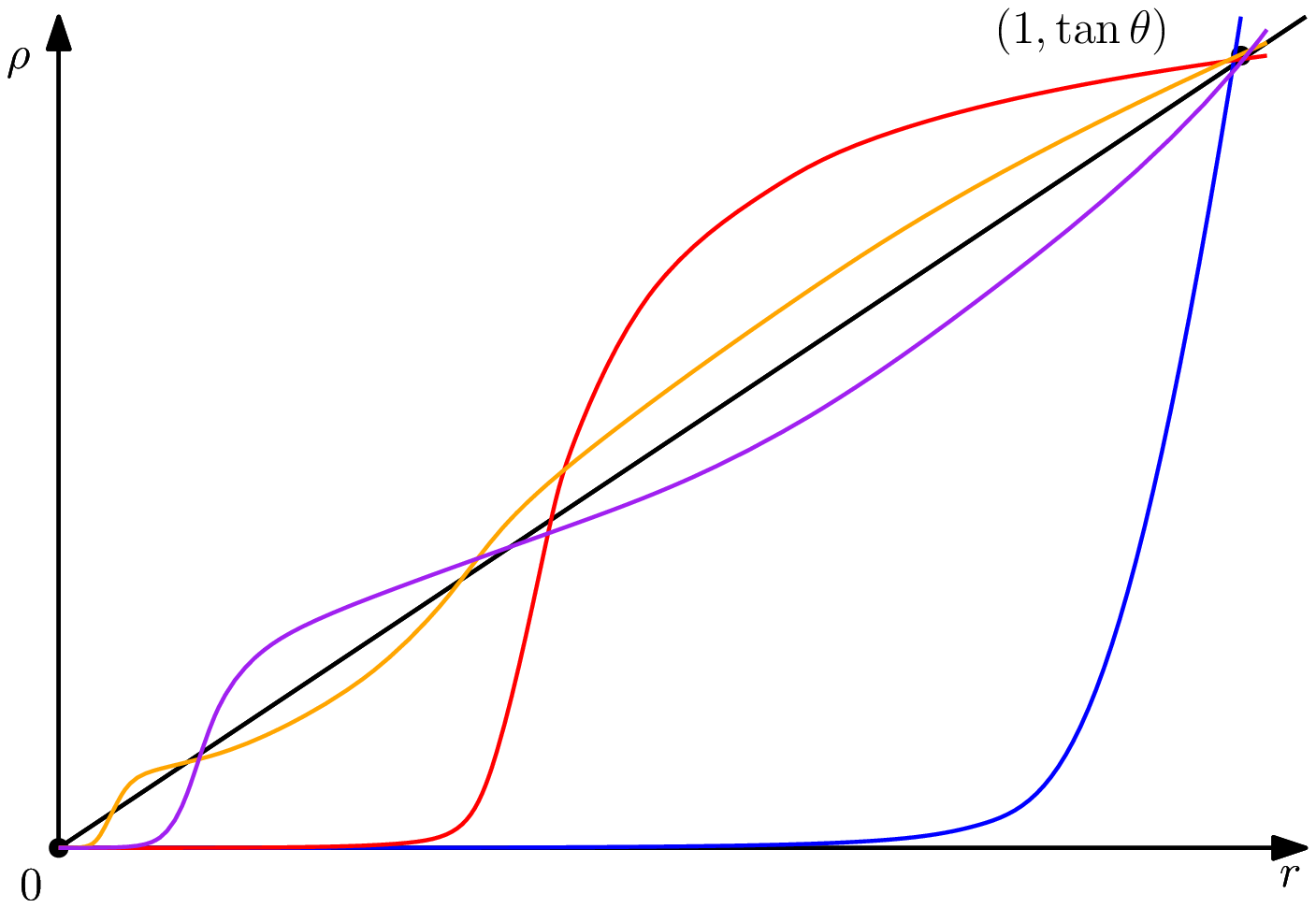}
          \caption{Curves representing analytic solutions}
          \label{AS}
         \end{figure}
         {\ }\\
         and the boundary of each minimal graph is exactly the graph of $\phi=\tan\theta\cdot f$.


\Section{Canonical metric on the quotient space}{Canonical metric on quotient space}\label{S3}%

According to \eqref{Mrf}, we restrict ourselves to the following $(n+2)$-dimensional smooth submanifold  associated to LOMSEs for both Type {\bf (I)} and Type {\bf (II)} 
               \begin{equation}\label{W}
               W:=\{(rx,\rho f(x)):\ x\in S^n,\ r\in \mathbb R_+,\ \rho\in \mathbb R\}
               \end{equation}
               in $\mathbb R^{m+n+2}$. 
               
               Every smooth function $\rho(r)$
               (or more generally, a parametric curve with nowhere vanishing velocity in the $r\rho$-plane)
               would give an embedded (or immersed) hypersurface in $W$.
               Stimulated by the ideas of  \cite{b, h-l, l} for submanifolds of low cohomogeneity,
               we consider the right half $r\rho$-plane as the ``quotient space" of $W$.
               Namely, the quotient space only reads off values $(r,\rho)$ for points in $W$.
         \begin{pro}\label{p1}
         With respect to canonical metric
               \begin{equation}\label{g0}
               g_0:=\sigma_0^2\cdot g=\sigma_0^2\cdot 
               \left[
               \left(
               r^2+\lambda^2\rho^2
               \right)^p
               \cdot
               r^{2(n-p)}
               \right]
               \cdot [dr^2+d\rho^2]
               \end{equation}
               where $\sigma_0$ is the volume of the $n$-dimensional unit sphere,
               the length of every curve
                 equals the volume of corresponding hypersurface in $W\subset \mathbb R^{m+n+2}$.
                        \end{pro}
\begin{proof}        
                         Note that this situation is similar to the cohomogeneity one case.
                         We only need to figure out the volume function $V(r,\rho)$
                         of the $n$-dimensional submanifold corresponding to $(r,\rho)$.
                         Then 
                         $$g_0=\big(V(r,\rho)\big)^2\cdot \left[dr^2+d\rho^2\right]$$ 
                         will have the desired property.
                         
                         For expression of $V(r,\rho)$,
                         let us fix a point $x\in S^n$ and
                        choose an orthonormal frame 
                        $\{e_1,\cdots, e_p,\ e_{p+1},\cdots, e_n\}$
                         of $T_xS^n$
                         such that
                         $\{f_*(e_1),\cdots,\ f_*(e_p)\}$ form an orthogonal set in $T_{f(x)}S^m$ 
                         of length $\lambda$
                         and $f_*(e_{p+1})=\cdots=f_*(e_n)=0$, 
                         by Definition \ref{LOMSE}. 
                         Set $f_i=\frac{f_*(e_i)}{\lambda}$ 
                         for $i=1,\cdots, p$.
                   Hence, the differential $DF_x$ for $F(x)= (rx,\rho f(x))$ sends
                         \begin{eqnarray*}
                         e_i&\longmapsto&(re_i,\ \lambda\rho f_i),\quad 1\leq i\leq p\\
                         e_{j}&\longmapsto&(re_{j},\ 0),\quad p+1\leq j\leq n.
                         \end{eqnarray*}
                         Thus, 
                          \begin{eqnarray*}
                         V(r,\rho)&=&\int_{F(S^n)}dv_{F(S^n)}
                         =\int_{S^n}\sqrt{\det(DF^T\cdot DF)}dv_{S^n}\\
                         &=&\int_{S^n}(\sqrt{r^2+\lambda^2\rho^2})^p\cdot r^{(n-p)}dv_{S^n}\\
                         &=&\sigma_0({r^2+\lambda^2\rho^2})^{\frac{p}{2}}\cdot r^{(n-p)},
                          \end{eqnarray*}
           and the proposition gets proved.
\end{proof}

\begin{rem}
              All solution curves exhibited in Figure \ref{AS} are geodesics.
                In particular, 
                $\rho=\tan\theta\cdot r$ is a geodesic with respect to $g$.
\end{rem}

{\ }

\section{Geodesic equation and the minimal surface equation}\label{S4}

We show in this section that the geodesic equation with respect to the metric 
$$
g = \left[
               \left(
               r^2+\lambda^2\rho^2
               \right)^p
               \cdot
               r^{2(n-p)}
               \right]
               \cdot [dr^2+d\rho^2]
               =u
               \cdot [dr^2+d\rho^2]
$$
gives rise to the Evolution Equation \eqref{ODE1}. 

The formula for the Christoffel symbol of a conformal metric is 
$$
\tilde \Gamma_{ij}^k = \Gamma_{ij}^k + \frac{1}{2}\Big((\log u)_i \delta^{k}_j + (\log u)_j \delta^k_i - (\log u)_k \delta_{ij}\Big).
$$
Direct calculation gives 
\begin{align*}
A &:= \frac{1}{2}(\log u)_r = \frac{(n-p)\lambda^2\rho^2 + nr^2}{(\lambda^2\rho^2 + r^2)r} = \Gamma_{rr}^r = \Gamma_{r\rho}^\rho = -\Gamma_{\rho \rho}^r,\\
B &:= \frac{1}{2}(\log u)_\rho =  \frac{p\lambda^2 \rho}{\lambda^2\rho^2 + r^2} = -\Gamma_{rr}^\rho = \Gamma_{r\rho}^r = \Gamma_{\rho\rho}^\rho.
\end{align*}
Let the arc-length parameter be $s$, and assume $r=r(s)$, $\rho = \rho(s)$ to be a geodesic. 
Denote $\dot r = \dfrac{dr}{ds}$, $\dot \rho = \dfrac{d\rho}{ds}$. The geodesic equation is then 
\begin{gather*}
\ddot r + A\dot r^2 + 2B\dot r\dot \rho - A\dot \rho^2 = 0,\\
\ddot \rho - B \dot r^2 + 2A\dot r\dot \rho + B\dot \rho^2=0.
\end{gather*}

Since
$
\rho_r = \dfrac{\dot \rho}{\dot r}, 
$
it follows that 
\begin{align*}
\rho_{rr} &= \frac{\ddot \rho \dot r - \dot \rho \ddot r}{\dot r^3}\\
&= B - A\rho_r + B\rho_r^2 - A\rho_r^3\\
&= (1+\rho_r^2) (B- A\rho_r)
\end{align*}
which can be verified to be \eqref{ODE1}. 

Furthermore, we remark that a rescaling of a geodesic is again geodesic.

{\ }

\Section{Computations on stability and instability}{Computations on stability and instability}\label{S5}
             
             In \cite{x-y-z} it has been shown that every LOC of $(n,p,2)$-type (subset of Type {\bf (I)}) is area-minimizing, thus stable minimal;
             whereas LOCs for Type  {\bf (II)}  LOMSEs are not area-minimizing. 
             The lengths of the solution geodesics (with more and more oscillations) in Figure \ref{AS} increase  to that of the LOC segment
             connecting the origin and $(1,\tan\theta)$ (open at the origin and closed at the other side).
             However, these geodesics are isolated, not forming a continuous family of geodesics.
             So it was unclear if an LOC segment for Type {\bf (II)} is stable or not.
             In this section, we determine stability and instability of these LOC segments
              in the quotient spaces.
             
  \begin{thm}\label{t3}
  The LOC segment for an LOMSE is stable if and only if  it is of Type {\bf (I)} and unstable if and only if it is of Type {\bf (II)}.
  \end{thm}
   \begin{rem}
   Theorems \ref{t1} and \ref{t2} are corollaries.
   \end{rem}
             
             \begin{proof} 
             The strategy is to study the Jacobi field equation. 
             We shall show that an LOC segment $\{ (r,\tan\theta\cdot r): r\in(0,1] \}$ contains conjugate points to the point $(1, \tan\theta)$ for Type {\bf (II)} and no conjugate points for Type {\bf (I)}. 
             
             Let us first consider the Gaussian curvature $K$ for metric $g$ in Proposition \ref{p1} (ignoring the factor $\sigma_0$). 
             A well-known formula for Gaussian curvature 
             for isothermal metric $h(dr^2+d\rho^2)$ is
             $$
             K=-\dfrac{1}{2h}\triangle \log h.
             $$
             By $h= \left(
               r^2+\lambda^2\rho^2
               \right)^p
               \cdot
               r^{2(n-p)}
              $,
              it becomes
        \begin{eqnarray*}
             K&=&-\dfrac{1}{2(r^2+\lambda^2\rho^2)^p\cdot r^{2(n-p)}}\triangle
             \left(
             p\log(r^2+\lambda^2\rho^2)+2(n-p)\log r
             \right)\\
             &=&\dfrac{1}{(r^2+\lambda^2\rho^2)^p\cdot r^{2(n-p)}}
             \left[
             \dfrac{n-p}{r^2}-p(\lambda^2-1)\dfrac{r^2-\lambda^2\rho^2}{(r^2+\lambda^2\rho^2)^2}
             \right].
       \end{eqnarray*}
             Along the geodesic $\{\rho=\tan\theta\cdot r\}$,
             \begin{equation}\label{sK}
             K(r)=\dfrac{1}{(1+\lambda^2\tan^2\theta)^p\cdot r^{2(n+1)}}
             \left[
             (n-p)-
             \dfrac{p(\lambda^2-1)}{(1+\lambda^2\tan^2\theta)^2}\left(1-\lambda^2\tan^2\theta\right)
             \right].
             \end{equation}
             To simplify the above expression, 
             we transform \eqref{sv} for LOMSE
             \begin{equation}\label{lpn}
            \dfrac{n-p}{\cos^2\theta}+\dfrac{p}{\cos^2\theta+\lambda^2\sin^2\theta}=n
             \end{equation}
             (by splitting $n=(n-p)+p$) to
             $$
             \dfrac{(n-p)\sin^2\theta}{\cos^2\theta}+
             \dfrac{p(\sin^2\theta-\lambda^2\sin^2\theta)}{\cos^2\theta+\lambda^2\sin^2\theta}=0,
             $$
             which implies
             \begin{equation}\label{tlpn}
(n-p)=\dfrac{p(\lambda^2-1)}{1+\lambda^2\tan^2\theta}.
             \end{equation}
             Now \eqref{sK} gives
             \begin{equation}\label{sK2}
             \begin{split}
             K(r)&=\dfrac{n-p}{(1+\lambda^2\tan^2\theta)^p\cdot r^{2(n+1)}}
             \left[
             1-\frac{1-\lambda^2\tan^2\theta}{1+\lambda^2\tan^2\theta}
              \right]\\
              &= \dfrac{2(n-p)}{(1+\lambda^2\tan^2\theta)^p\cdot r^{2(n+1)}}
             \left[
             1-\frac{1}{1+\lambda^2\tan^2\theta}
              \right].
              \end{split}
             \end{equation}
             
             For solutions to the Jacobi equation, we need the arc-length parameter 
                   \begin{eqnarray}
s&=&\int_{0}^r (1+\lambda^2\tan^2\theta)^{\frac{p}{2}}r^{n}\cdot \sqrt{1+\tan^2\theta}\ dr\nonumber\\
&=& \dfrac{\sqrt{1+\tan^2\theta} \ (1+\lambda^2\tan^2\theta)^{\frac{p}{2}}}{n+1}\cdot r^{n+1}\label{our t}. 
                   \end{eqnarray}
                   Therefore
                  \begin{eqnarray*}
                   K(s)&=&
                     \dfrac{2(n-p)(1+\tan^2\theta)}{(n+1)^2s^2}
                   \left [
                                   1-\frac{1}{1+\lambda^2\tan^2\theta}    
                     \right].
                     \end{eqnarray*}
                     Note that \eqref{tlpn} implies
                   \begin{align*}
                                   1-\frac{1}{1+\lambda^2\tan^2\theta} &= \frac{p\lambda^2 -n}{p(\lambda^2-1)} \\
\text{and}\quad\quad                    1+\tan^2\theta &= \frac{n(\lambda^2 -1)}{\lambda^2(n-p)}.
                       \end{align*}
                  So      
                  $$   
                   K(s) = \frac{2(p\lambda^2-n)n}{p\lambda^2(n+1)^2s^2}=\dfrac{2(k^2+kn-k-n)n}{(k+n-1)k(n+1)^2}\cdot \dfrac{1}{s^2},
                  $$
                      according to \eqref{la}.           
             
                  Define 
                  \begin{equation}\label{def a}
                  a:=\dfrac{2(k^2+kn-k-n)n}{(k+n-1)k(n+1)^2}.
                  \end{equation}
                  Fix a unit normal vector field $N$ along the LOC segment.
                   Then $J(s)N$ is a Jacobi vector field if and only if
                   $$J''(s)+K(s)J(s)=0.$$
                   This is a Euler equation. 
                   Consider $J(s)=s^l$. Then
                   $$
                   l(l-1)+a=0,
                   $$
                   and the solutions are
                   \begin{equation}\label{sJ}
                   J(s)=
                   C_1e^{\frac{1+\sqrt\triangle}{2}\log s}+C_2e^{\frac{1-\sqrt\triangle}{2}\log s}.
                   \end{equation}

                   {\bf (A).} When $\triangle=1-4a>0$, 
                   all solutions
                   $J(s)=
                   C_1s^{\frac{1+\sqrt\triangle}{2}}+C_2s^{\frac{1-\sqrt\triangle}{2}}$
                   are linear combinations of two power functions whose exponents are distinct real numbers.
                   Apparently for $J$ to vanish at two distinct $s_1, s_2>0$, we must have $C_1=C_2=0$. 
                   Hence the index of the LOC segment is zero and the segment is stable.
                   Similarly, when $\Delta = 0$, then $J(s) = C_1 \sqrt{s} + C_2 \sqrt{s}\log s$ can not vanish at two distinct $s_1, s_2$ unless it is trivial. 
                    
                    {\bf (B).} When $\triangle=1-4a<0$,  all solutions are
%
{
                   $$J(s)=
                   C_1\sqrt s \cos
                   \left(  
                   \dfrac{\sqrt {-\triangle}}{2}\log s
                   \right)
                   +
                   C_2
                   \sqrt s
                   \sin
                   \left(  
                   \dfrac{\sqrt {-\triangle}}{2}\log s
                   \right).
                   $$
                 It is clear that there are infinitely many
                   conjugate points to any point 
                   on the LOC segment. 
                   Therefore, the LOC segment is unstable.
                   }
                   
                   Now let us consider {\bf (A)} first by solving
                    $$a=\dfrac{2(k^2+kn-k-n)n}{(k+n-1)k(n+1)^2}<\dfrac{1}{4}$$
                   which is equivalent to
                                      \begin{eqnarray*}
8n(k^2+kn-k-n)&<& (k+n-1)k(1+n)^2\\
\Longleftrightarrow\ \  8nk^2+8n(n-1)k -8n^2&<& k^2(1+n)^2+k(n-1)(n+1)^2
                   \end{eqnarray*}
and further transformed to
                  \begin{eqnarray}
                  ((1+n)^2-8n)k^2+k(n-1)[(n+1)^2-8n]+8n^2&>&0 \nonumber\\
\Longleftrightarrow\quad                
 (n^2-6n+1)k^2+k(n-1)(n^2-6n+1)+8n^2&>&0 \label{eqnk}
                   \end{eqnarray}
                   If $n\geq6$, then the leading coefficient of $k^2$ is positive and 
                   \eqref{eqnk} holds for all positive $k$.
                  Note that positive integer $n$ has to be odd as the dimension of source space of an Hopf fibration 
                  and positive integer $k$ is always even (see Remark \ref{evenk}).
                  Direct computation shows that
                  for $n=5$  the solutions are $k=2,4$;
                  for $n=3$ the only solution is $k=2$.
                  It turns out that $a=\frac{1}{4}$ has no solution in our situation. 
                  
                 Therefore, the solution situation of {\bf (B)} is exactly the opposite to {\bf (A)}.
                  
                  By comparing with the $(n, k)$ values of Type {\bf {\bf (I)}} and Type {\bf (II)},
                  we have shown that the LOC segment is stable if and only if it is of Type {\bf (I)};
                  and unstable if and only if it is of Type {\bf (II)}.
                  \end{proof}

 \begin{rem}
              As pointed out in \cite{Simons}, in general a (compactly supported) Jacobi field along a surface can not generate a family of nearby minimal surfaces correspondingly.
              In our case, the Jacobi field along the LOC segment  has a lifting to an ``upstair" Jacobi field along the regular part of the minimal cone LOC in the Euclidean space,
              and it can produce a family of nearby minimal surfaces by lifting geodesics around the LOC segment. 
\end{rem}

{\ }

\section{Some remarks on instability}\label{S6}

In this section, we  further discuss  the close relationship of above calculations with that in \cite{x-y-z0}, where conditions for Type {\bf (II)} Lawson-Osserman cones were discovered. 
The Jacobi field method here is derived from the linearization of geodesic equation;
while, in \cite{x-y-z0}, the ODE requirement \eqref{ODE1} for a minimal surface is 
transformed to a dynamic system with 
linearization analysis at a fixed point that corresponds to the LOC. 
We shall see that the periods of Jacobi fields observed in current paper coincide with those of dynamic behaviors in \cite{x-y-z0},
and moreover, we shall establish an explicit translation between these two perspectives.

{
By \eqref{our t}, the arc-length $s$ is a constant multiple of $r^{n+1}$. 
Therefore, solution \eqref{sJ} in variable $r$ becomes 
\begin{equation}\label{preJ}
J(r) = \tilde C_1 e^{\frac{1+\sqrt{\Delta}}{2} (n+1)\log r} + \tilde C_2 e^{\frac{1-\sqrt{\Delta}}{2} (n+1)\log r}.
\end{equation}
From \eqref{def a}, the solutions are linear combinations of 
\begin{equation}\label{our cal}
\exp\left(\frac{1}{2}\bigg((n+1) \pm \sqrt{(n+1)^2 + 8n\Big(\frac{n}{k(n+k-1)}-1\Big)}\bigg)\log r\right).
\end{equation}
}

In  \cite{x-y-z0}, the ODE \eqref{ODE1} for minimal graphs is rewritten as 
\begin{equation}\label{ODE2}
\aligned
\left\{\begin{array}{ll}
\varphi_t=\psi,\\
\psi_t=-\psi-\Big[\big(n-p+\f{p}{1+\la^2\varphi^2}\big)\psi+\big(n-p+\f{(1-\la^2)p}{1+\la^2\varphi^2}\big)\varphi\Big]\big[1+(\varphi+\psi)^2\big].
\end{array}\right.
\endaligned
\end{equation}
 in variables
$$
t:= \log r, \quad \varphi:= \frac{\rho}{r}, \quad \psi:= \varphi_t.
$$
       Then the point $(\varphi, \psi)=(\tan \theta, 0)$ is a stationary point of the ODE system.
       Denoting $$\varphi_0 = \tan \theta=\sqrt{\dfrac{p-n\lambda^{-2}}{n-p}},$$
        the linearization of the ODE system at the point is 
       $$
       \begin{pmatrix}
       (\phi - \varphi_0)_t\\
       \psi_t
       \end{pmatrix}
       = 
       B
       \begin{pmatrix}
       (\phi - \varphi_0)\\
       \psi
       \end{pmatrix},
       \quad 
       B = \begin{pmatrix}
       0 & 1\\
       2n\big(\frac{n}{k(k+n-1)}-1\big) & -(n+1)
       \end{pmatrix}.
       $$
       The eigenvalues of $B$ are the solutions of 
       $$
       \lambda^2 + (n+1)\lambda - 2n\Big(\frac{n}{k(k+n-1)}-1\Big) = 0
       $$
       and hence are 
       \begin{equation}\label{l12}
       \lambda_{1, 2} = {\frac{1}{2}\bigg(-(n+1)\pm \sqrt{(n+1)^2 + 8n\Big(\frac{n}{k(k+n-1)}-1}\Big)}\bigg).
       \end{equation}
       In our case, these two eigenvalues never equal. Therefore, following standard procedure, we use the matrix $P$ of the corresponding eigenvectors to get 
       $$ B  = P\begin{pmatrix}
       \lambda_1 & \\
       & \lambda_2
       \end{pmatrix} P^{-1}.
       $$
       Although $\lambda_{1, 2}$ and $P$ may be complex, the product on the right hand side is real.
       Then every solution to the linearization of  system  \eqref{ODE2} at $(\varphi_0,0)$ is
    \begin{equation}\label{decom}
       P\begin{pmatrix}
       e^{\lambda_1 t} & \\
       & e^{\lambda_2 t}
       \end{pmatrix} P^{-1} \big(\vec v_0-(\varphi_0,0)\big)+(\varphi_0,0),
     \end{equation}
       for some initial point 
 $\vec v_0$. 
 The product term,
       a linear combination of  
       \begin{equation}\label{b4 cal}
       e^{\lambda t} = \exp\left({\frac{1}{2}\bigg(-(n+1)\pm \sqrt{(n+1)^2 + 8n\Big(\frac{n}{k(k+n-1)}-1\Big)}\bigg)}\log r\right),
       \end{equation}
       is real.
       We see that when the square root is imaginary, the ODE system has a stable spiral singularity at $(\tan\theta, 0)$. 
       This is how \cite{x-y-z0} defines Type (II). 

The square root term in \eqref{b4 cal} is the same as that in 
\eqref{our cal}. 
This explains the coincidence of the periodicity.
Furthermore, for a more clear correspondance on 
the instability of Type (II) Lawson-Osserman cones,
               we shall give an explicit translation from \eqref{b4 cal} to \eqref{our cal} to show the intimate connection.
{\ }
        
              Theoretically, it is well known that \eqref{preJ} is obtained by linearization along $\{\rho=\varphi_0 r\}$;
              while \eqref{decom} is gained by linearization at $(\varphi_0,0)$ exactly corresponding to the LOC curve $\{\rho=\varphi_0 r\}$.
              This is the essential reason that makes it possible and natural 
              to derive Jacobi fields from orbits to \eqref{ODE2} in the $\varphi \psi$-plane around $(\varphi_0,0)$.
              
              We take 	$(n,p,k)=(5,4,6)$ for example with
              $$
       B = \begin{pmatrix}
       0 & 1\\
       -\frac{55}{6} & -6
       \end{pmatrix},
       \ \ \ \ \ 
        \lambda_{1, 2} = -3\pm \dfrac{i}{\sqrt 6},
        \ \ \ \ \ \text{and} 
       $$
         $P$ can be chosen to be
           $$
     \begin{pmatrix}
       \dfrac{1}{ -3+ \frac{i}{\sqrt 6}} &  \dfrac{1}{ -3- \frac{i}{\sqrt 6}}\\
       &\\
       1 & 1
       \end{pmatrix}.
        $$
       Based on \eqref{decom}, the orbit through point $(\varphi_0, -\epsilon)$  is
                  $$
     \begin{pmatrix}
       e^{-3t}\left(3\sqrt 6 \sin\dfrac{t}{\sqrt 6}+\cos\dfrac{t}{\sqrt 6}\right) &
       \sqrt 6 e^{-3t}\sin \dfrac{t}{\sqrt 6} \\
       &\\
       -\dfrac{55}{6}\sqrt 6 e^{-3t} \sin{\dfrac{t}{\sqrt 6}}& 
       e^{-3t}\left(3\sqrt 6 \sin\dfrac{t}{\sqrt 6}-\cos\dfrac{t}{\sqrt 6}\right)
       \end{pmatrix}\cdot
       \begin{pmatrix}
       0\\
       \\
       -\epsilon
       \end{pmatrix}
       +
          \begin{pmatrix}
       \varphi_0\\
       \\
       0
       \end{pmatrix}.
        $$
        
     Although the above expression is an approximate orbit to system \eqref{ODE2},
     it is exactly what we need to consider in the limiting procedure. 
     Note that
     by setting $r=e^t$ one has
     $
     \varphi=\varphi_0-\epsilon \sqrt 6 e^{-3t}\sin \dfrac{t}{\sqrt 6}.
     $
    Therefore,
    $$
     \Big\{\rho=\varphi r=\varphi_0  r-\epsilon\dfrac{\sqrt 6}{r^{2}}\sin \dfrac{\log r}{\sqrt 6}:\ \epsilon\in\mathbb R_+\Big\}
     $$
    provide a family of approximations of geodesics in  $C^1$ sense for $\epsilon\rightarrow 0+$ with the following picture
     \begin{figure}[h]
          \includegraphics[scale=0.8]{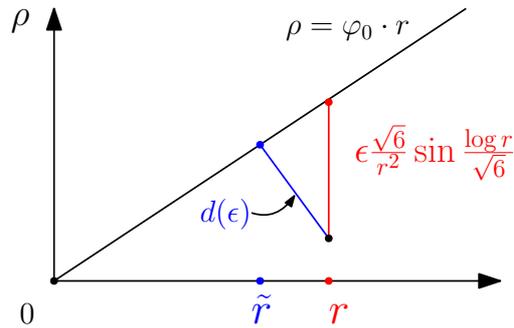}
          \caption{Illustration on the translation back to $r\rho$-plane}
         \end{figure}
     \\  
       which demonstrates 
       $
       r=\tilde r+\dfrac{\varphi_0}{\sqrt{1+\varphi_0^2}} d(\epsilon).
       $
       
       The Euclidean length of the red vertical segment leads to 
       \begin{equation}\label{Relation}
       \sqrt{1+\varphi_0^2}\cdot d(\epsilon)=\epsilon\sqrt 6 \cdot
       \dfrac{\sin\frac{\log\left(\tilde r+\frac{\varphi_0}{\sqrt{1+\varphi_0^2}} d(\epsilon)\right)}{\sqrt 6}}
       {\left(\tilde r+\frac{\varphi_0}{\sqrt{1+\varphi_0^2}} d(\epsilon)\right)^2},
       \end{equation}
       and further, the derivative of the expression in $\epsilon$ at zero gives
        \begin{equation}\label{Relation}
       \sqrt{1+\varphi_0^2}\cdot d'(0)=\sqrt 6 \cdot
       \frac{\sin\frac{\log\tilde r}{\sqrt 6}}
       {\tilde r^2}.
       \end{equation}
       
       Assume that $N$ is the pointing downward unit normal vector field along $\{\rho=\varphi_0 r\}$.
       By the formulation of $g_0$ in \eqref{g0},
       its Euclidean length is of order $-5$ in $\tilde r$.
       Hence, by virtue of the $C^1$ approximation, it produces the Jacobi field
       \begin{equation}\label{scalaroverN1}
       d'(0)=\left(C \tilde r^3\sin\frac{\log\tilde r}{\sqrt 6}\right) N, \text{ for some constant } C.
       \end{equation}
       Moreover, if at the beginning we set $r=e^{t-\frac{\sqrt 6\pi}{2}}$ instead, 
       then \eqref{scalaroverN1} becomes
          \begin{equation}\label{scalaroverN2}
       \tilde d'(0)=\left(\tilde C  \tilde r^3\cos\frac{\log\tilde r}{\sqrt 6}\right) N, \text{ for some constant } \tilde C.
       \end{equation}
       In such way, we see that linear combinations of \eqref{scalaroverN1} and \eqref{scalaroverN2} 
       exhaust all Jacobi fields \eqref{preJ} for this concrete example.
       
       In general, the same conclusion can also be obtained similarly.
       Assume that
           $$
     P=\begin{pmatrix}
       p_{11} &  p_{12}\\
       &\\
       p_{21} & p_{22}
       \end{pmatrix}
       \ \ \text{ and } \ \
       P^{-1}=\begin{pmatrix}
       p^*_{11} &  p^*_{12}\\
       &\\
       p^*_{21} & p^*_{22}
       \end{pmatrix}
       =
       \dfrac{1}{\det P}
       \begin{pmatrix}
       p_{22} &  -p_{12}\\
       &\\
       -p_{21} & p_{11}
       \end{pmatrix}.
        $$
       Now the approximate orbit through $(\varphi_0, -\epsilon)$  is
    $$
     P\begin{pmatrix}
       e^{\lambda_1 t} &  \\
       &\\
       & e^{\lambda_2 t}
       \end{pmatrix}
       P^{-1}
       \cdot
       \begin{pmatrix}
       0\\
       \\
       -\epsilon
       \end{pmatrix}
       +
          \begin{pmatrix}
       \varphi_0\\
       \\
       0
       \end{pmatrix}
 =
     \begin{pmatrix}
       \varphi_0\\
       \\
       0
       \end{pmatrix}
       -\epsilon
       \begin{pmatrix}
       e^{\lambda_1 t}
       p_{11} p^*_{12} + 
      e^{\lambda_2t}
       p_{12} p^*_{22} 
       \\
       \\
       e^{\lambda_1 t}
       p_{21} p^*_{12} + 
      e^{\lambda_2t}
       p_{22} p^*_{22} 
       \end{pmatrix}.
        $$
       Hence the corresponding curves in $r\rho$-plane satisfies 
       \begin{equation}\label{grhor}
       \rho=
       \big(
       \varphi_0-\epsilon
       (e^{\lambda_1 t}
       p_{11} p^*_{12} + 
      e^{\lambda_2t}
       p_{12} p^*_{22} )
       \big)
       r.
       \end{equation}
      By
      $p_{11} p^*_{12}+ p_{12} p^*_{22}=0$
     and \eqref{l12},
      we have 
      \begin{align*}
       \rho &=\big(\varphi_0-\epsilon p_{11}p^*_{12}(e^{\lambda_1t}-e^{\lambda_2t})\big)r\\
       &=\Big(\varphi_0-\epsilon p_{11}p^*_{12}(2e^{\Re \lambda_1t}\sin(\Im\lambda_1 t) i\big)\Big)r\\
       &=\Big(\varphi_0+\epsilon \frac{p_{11}p_{12}}{\det P} (2e^{\Re \lambda_1t}\sin(\Im\lambda_1 t) i\big)\Big)r.
                                                       \end{align*}
We claim that $\frac{p_{11}p_{12}}{\det P} i$ is a real nonzero number.
Note that
$$\det \begin{pmatrix}
       p_{11} &  \overline{p_{11}}\\
       &\\
       p_{21} & \overline{p_{21}}
       \end{pmatrix}
       =2\Im(p_{11}\overline{p_{21}}) i,
\ \ \text{ and } \ \
       \begin{pmatrix}
       \overline{p_{11}}\\
       \\
       \overline{p_{21}}\\
       \end{pmatrix}
       =\alpha\cdot
       \begin{pmatrix}
       p_{12}\\
       \\
       p_{22}
       \end{pmatrix}
       \ \ \text{ for some }\alpha\neq0 \in \mathbb C.
       $$
So
        \begin{align*}
       \dfrac{p_{11}p_{12}}{\det P} i
      & =\dfrac{\|p_{11}\|^2}{\alpha \det P}i
      =\dfrac{\|p_{11}\|^2}{2\Im(p_{11}\overline{p_{21}})}
      \in \mathbb R.
         \end{align*}
      
      Similar to the preceding example, by employing $r=e^t$,
      we have the following relation for signed $d(\epsilon)$
      \begin{equation}\label{Relation2}
       \sqrt{1+\varphi_0^2}\cdot d(\epsilon)
       = - \dfrac{\epsilon\|p_{11}\|^2}{\Im(p_{11}\overline{p_{21}})} \cdot
       \dfrac{\sin\left(\Im\lambda_1\log\left(\tilde r+\frac{\varphi_0}{\sqrt{1+\varphi_0^2}} d(\epsilon)\right)\right)}
       {\left(\tilde r+\frac{\varphi_0}{\sqrt{1+\varphi_0^2}} d(\epsilon)\right)^{\frac{n-1}{2}}},
       \end{equation}
       Since the Euclidean length of $N$ is of order $-n$ in $r$, 
       \begin{equation}\label{scalaroverN3}
       d'(0)=\left(C \tilde r^{\frac{n+1}{2}}
       \sin\big(\Im\lambda_1\log\tilde r\big)
       \right) N, \text{ for some constant } C.
       \end{equation}
      By using $r=e^{t-\frac{\pi}{2\Im\lambda_1}}$ (which generates a rescaling on the $r\rho$-plane) instead, 
    \eqref{scalaroverN3} becomes
          \begin{equation}\label{scalaroverN4}
       {\tilde d}'(0)=
       \left(\tilde C \tilde r^{\frac{n+1}{2}}
       \cos\big(\Im\lambda_1\log\tilde r\big)
       \right) N, \text{ for some constant } \tilde C.
              \end{equation}
Thus linear combinations of \eqref{scalaroverN3} and \eqref{scalaroverN4} 
      provide  all Jacobi fields \eqref{preJ}.
      
{\ }

\section{Some remarks on stability}\label{S7}
In this section, we shall show that the LOC curve
for {Type \bf{(I)}} in the quotient space is not merely stable minimal but in fact length-minimizing, 
namely a ray in the sense of classical Riemannian geometry, in its certain angular neighborhood.

To explain local length-minimality, i.e. being length-minimizing in some neighborhood of the LOC curve,
we shall construct a geodesic foliation in some angular neighborhood around the LOC curve.
\begin{figure}[h]
          \includegraphics[scale=0.62]{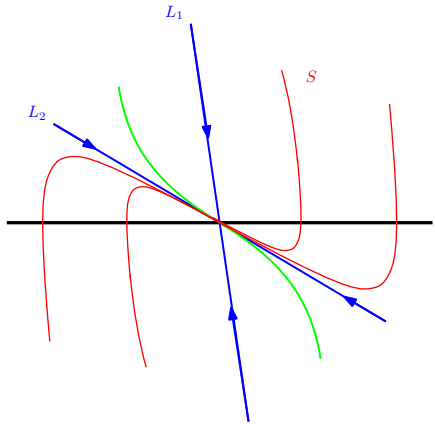}
          \caption{Infinitesimal orbits near $(\varphi_0,0)$ for {Type (I)} in the $(\varphi,\psi)$-plane}
         \label{F2}
         \end{figure}
Note that $\Delta=(n+1)^2 + 8n\big(\frac{n}{k(n+k-1)}-1\big)>0$. 
Let $L_{1,2}$ be the line through $(\varphi_0,0)$ with slopes 
         $\frac{1}{2}[-(n+1)\mp \sqrt{\Delta}]$.
         Apparently, all orbits accumulate to $L_2$ expect one to $L_1$ in the infinitesimal model.
Hence, starting from $(\varphi_1,0)$ where $\varphi_1$ is sufficiently close to $\varphi_0$ and $\varphi_1>\varphi_0$, 
there exists an orbit of \eqref{ODE2} limiting to $(\varphi_0,0)$ below the $\varphi$-axis with decreasing $\varphi$ values in $(\varphi_0,\varphi_1]$.
This orbit corresponds to a geodesic curve $\gamma_1$ in the $r\rho$-plane which has strictly decreasing slopes within $(\varphi_0,\varphi_1]$.
%
%
Similar discussion works for the other side (with $\varphi_2<\varphi_0$).
                                                                           \begin{figure}[h]
          \includegraphics[scale=0.45]{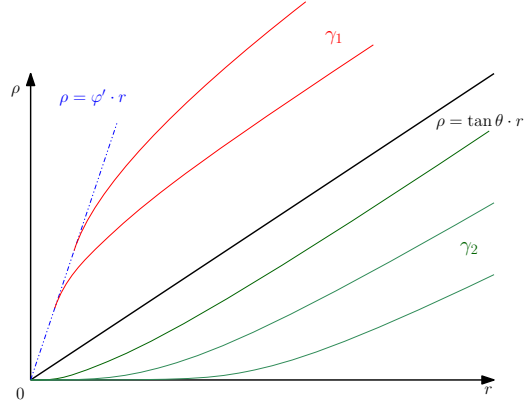}
          \caption{Geodesic foliation around the ray $\rho=\tan\theta\cdot r$}
         \end{figure}
However it is convenient to use the unique orbit from $(0,0)$ to $(\varphi_0,0)$, i.e., $\varphi_2=0$, with increasing $\varphi\in (0,\varphi_0)$ explored in \cite{x-y-z0},
which corresponds to a geodesic curve $\gamma_2$ in the $r\rho$-plane with strictly increasing slope within $(0,\varphi_0)$.
Thus we can gain
 a foliation of homothetic geodesics (i.e., obtained by dilations of $\gamma_1$ and $\gamma_2$) in the angular region
$\mathscr S$ between $\{\rho=0\}$ and $\{\rho=\varphi_1\cdot r\}$.

Therefore, the LOC curve is length-minimizing in the angular region $\mathscr S$ following a standard calibration argument (the fundamental theorem of calibrated geometries) 
using the calibration form
$\omega=v \lrcorner\ \tau$ where $v$ is the unit tangent vector fields along the oriented foliation and $\tau$ the oriented unit volume form of $\mathscr S$,
see \cites{b-d-g, HL2, HL1} for details.

{\ }

\end{document}